\documentclass[a4paper,12pt]{article}

\usepackage{cite}
\usepackage[dvips]{graphicx}
 \usepackage{amsmath}
 \usepackage{array}
 \usepackage{amsfonts, amssymb, amsbsy}
 \usepackage{amsthm}

\usepackage{tabularx}
\usepackage{xcolor}

\newtheoremstyle{iet}
  {12pt}
  {\topsep}
  {\normalfont}
  {0pt}
  {\normalfont\itshape}
  {:}
  {3pt}
  {\thmname{#1}\thmnumber{ #2}\thmnote{#3}}

\theoremstyle{iet}

 \newtheorem{theorem}{Theorem}
 
 \newtheorem{definition}{Definition}
 \newtheorem{example}{Example}
 \newtheorem{remark}{Remark}
 \newtheorem{lemma}{Lemma}

\allowdisplaybreaks

\makeatletter
\renewcommand{\maketitle}{\bgroup\setlength{\parindent}{0pt}

\begin{flushleft}
  \textbf{\Large\@title}

\vspace*{1cm}
  \@author
\end{flushleft}\egroup
} \makeatother \makeatletter
\renewcommand\@biblabel[1]{#1.}
\makeatother

\usepackage{amsfonts,amssymb,amsbsy}
 \usepackage{latexsym}
 \usepackage{amsmath}

\begin{document}




\title{State bounding for positive coupled differential - difference equations with bounded disturbances}

\author{Phan Thanh Nam$^{1}$, Thi-Hiep Luu$^1$\\
$^1$ \textit{Department of Mathematics, Quynhon University, Binhdinh, Vietnam}
}

\maketitle

%

\noindent\textbf{Abstract:}
In this paper, the problem of finding state bounds is considered, for the first time, for a class of positive time-delay coupled differential-difference equations (CDDEs) with bounded disturbances. First, we present a novel method, which is based on nonnegative matrices and optimization techniques, for computing a  like-exponential componentwise upper bound of the state vector of the CDDEs without disturbances. The main idea is to establish bounds of the state vector on finite-time intervals and then, by using the solution comparison method and the linearity of the system,  extend to infinite time horizon. Next, by using
state transformations, we extend the obtained results to a class of CDDEs with bounded disturbances. As a result, componentwise upper bounds, ultimate bounds and invariant set of the perturbed system are obtained. The feasibility of obtained results are illustrated through a numerical example.

%
%
%
%
%
%
%
%
%
%
%
%
%
%
%

\section{Introduction}\label{sec1}
Coupled differential-difference equations (CDDEs) are dynamical systems,
 which include a differential equation coupled with a difference equation.
Due to the fact that there are many systems in engineering  such as electrical systems, fluid systems, neutral systems and so on, which are described by  (or reformulated into) CDDEs, the  stability problem for classes of CDDEs has attracted much research attention during the past decades \cite{Niculescu:01,Rasvan:06,Pepe:08,Gu:09,Gu:10,Li:10,Gu:11,Shen:15,Ngoc:18,Pathirana:18}.
The most widely used method to this problem is based on the Lyapunov-Krasovskii functionals combining with linear matrix inequality technique (see \cite{Pepe:08,Gu:09,Gu:10,Gu:11,Li:10} and the reference therein). Recent, by exploiting
properties Metzler/Schur matrices combining with the comparison method, the authors \cite{Shen:15}, for the first time, presented another method and reported a new result on stability of a class of positive CDDEs with bounded time-varying delays. The topic on stability analysis of positive CDDEs by using the second method has gained an quickly increasing research attention in the recent years \cite{Ngoc:18,Pathirana:18,Cui:18a,Cui:18b,Sau:16,Sau:18}.

Disturbances are usually unavoidable in practical engineering systems due to many reasons such as external noises, measurement errors, modeling inaccuracies, linear approximation and so on. In the presence of disturbances, in general, it is hard to achieve asymptotic stability for perturbed dynamical systems. However, under the assumption that disturbances are bounded by a known bound, the bounded-ness/the convergence within a bounded set can be guaranteed. Hence, the topic on  state bounding/reachable set bounding/robust convergence  for classes dynamical systems perturbed by unknown-but-bounded disturbances has been an important issue in control theory and has attracted significant research attention during the past decades \cite{Boyd:94,Khalil:02,Fridman:03,Kofman:07,Kwon:11,Haimovich:13,Feng:15,Lam:15,Du:16,
Nam:15a,Nam:15b,Nam:16a,Nam:16b,Thuan:17,Li:17,Liu:18,Trinh:15,Trinh:18}.  There are two commonly used approaches to this problem. For classes of linear systems whose matrices are constant, the most widely used approach  is based on the Lyapunov method combining with linear matrix inequality technique \cite{Fridman:03,Kwon:11,Feng:15,Lam:15,Nam:15a,Thuan:17,Li:17,Liu:18,Trinh:15,Trinh:18}. For classes of positive linear systems, the   second widely used approach is based on properties of Metzler/Schur matrices combining with the solution comparison method \cite{Kofman:07,Haimovich:13,Du:16,Nam:15b,Nam:16a,Nam:16b}. It is worthy to note that the second approach combining with the comparison method is also very useful for classes of nonlinear/time-varying systems which are bounded by positive linear systems \cite{Nam:15b,Nam:16a,Nam:16b}.

By using the first approach, the authors \cite{Feng:15} reported an result on reachable set bounding for a class of perturbed time-delay singular systems, which includes class of CDDEs as a special case. Later, some its extensions to more general classes of singular systems has also given in \cite{Li:17,Liu:18}. However, so far, there has not been any result, which is obtained by using the second approach, on state bounding for positive CDDEs/singular systems with bounded disturbances. Motivated by this, in the paper, we study about the  problem of finding state bounds for a class of positive CDDEs with bounded disturbances by using the second approach. We present a novel method to derive componentwise state bounds for the positive CDDEs with bounded disturbances,including three main steps: (i) finite-time convergence for linear positive systems; (ii) like-exponential componentwise upper bound for CDDEs without disturbances;  and  (iii) state bounding for CDDEs with bounded disturbances.

The paper is organized as follows. After the introduction, the
problem statement and preliminaries are introduced in Section 2.
The main results are given in Section 3. A numerical example with
simulation results are given in Section 4.
 Finally, a conclusion is drawn in Section 5.
\section{Problem statement and preliminaries}
\noindent{\em Notations:}
$\mathbb{R}^n
(\mathbb{R}_{0,+}^n, \mathbb{R}_{+}^n)$ is the
$n$-dimensional (nonnegative, positive) vector space;
 $e_i=[0_{1\times(i-1)}$ $1 \ 0_{1\times(n-i)}]^T\in \mathbb{R}^n$
 is the $i^{th}$-unit vector in $\mathbb{R}^n$;
 $l_j=[0_{1\times(j-1)}$ $1 \ 0_{1\times(m-j)}]^T\in \mathbb{R}^m$
 is the $j^{th}$-unit vector in $\mathbb{R}^m$; $\overline{1,n}=\{1,2,\cdots,n\}$;
given three vectors $x=[x_1\ x_2\ \cdots\ x_n]^T \in \mathbb{R}^{n},$
$y =[y_1\ y_2\ \cdots\ y_n]^T \in \mathbb{R}^{n}$, $q=[q_1\ q_2\
\cdots\ q_n]^T \in \mathbb{R}_{0,+}^{n},$ two $n\times
n$ matrices $A=[a_{ij}]$ and $B=[b_{ij}]$, the following notations will
be used in our development:   $x \prec y (\preceq y)$ means that
$x_i<y_i(\leq y_i),\forall i\in\overline{1,n}$; $A \prec B (\preceq
B)$ means that $a_{ij}< b_{ij} (\leq b_{ij}),\forall
i,j\in\overline{1,n}$;  $A$ is nonnegative if $0 \preceq A$; $A$ is
a Metzler matrix if $a_{ij}\geq 0, \forall i,j\in\overline{1,n},
i\not=j$; $\mathcal{B}(0,q)=\{ x\in
\mathbb{R}^{n}_{0,+}:\ x\preceq q\}$ is an orthotope (hyperrectangle) in
$\mathbb{R}^{n}_{0,+}$;
 $s(A)=\max\{Re(\lambda): \lambda \in \sigma(A)\}$ stands for the spectral abscissa of
 a matrix $A$; $\rho(A)=\max\{|\lambda|: \lambda \in \sigma(A)\}$
stands for the spectral radius of a matrix $A$. $A$ is
a Schur matrix if $\rho(A)<1$;
The computations, such as minimum, maximum of a set of finite vectors, limitation of  a vector-valued function, etc., are understood in the component-wise sense.\\

Consider the linear CDDEs with bounded time-varying delays
\begin{align}\label{1}
\begin{split}
\dot{x}(t)&=Ax(t)+By(t-h_1(t))+\omega(t),\ t\geq t_0\geq 0, \\
y(t)&=Cx(t)+Dy(t-h_2(t))+d(t),
\end{split}
\end{align}

where $x(.)\in \mathbb{R}_{0,+}^n$, $y(.)\in \mathbb{R}_{0,+}^m$ are the state
vectors. Matrices $A\in \mathbb{R}^{n\times n}$, $B\in \mathbb{R}_{0,+}^{n\times m}$,
$C\in \mathbb{R}_{0,+}^{m\times n}$ and $D\in \mathbb{R}_{0,+}^{m\times m}$ are known. $D$ is assumed to be a Schur matrix. The disturbance vectors $\omega(.)\in \mathbb{R}_{0,+}^n$, $d(.)\in \mathbb{R}_{0,+}^m$ are unknown but assumed to be bounded by known bounds, i.e.,
\begin{align}
0&\preceq \omega(t)\preceq \overline{\omega},\ \ \forall t\geq t_0, \label{2}\\
0 &\preceq d(t)\preceq \overline{d},\ \ \forall t\geq t_0,\label{3} \vspace{-0.2cm}
\end{align}
where $\overline{\omega},\overline{d}$ are two known vectors.
The unknown time-varying delays,
$h_1(.)\in \mathbb{R}_{0,+}$, $ h_2(.)\in \mathbb{R}_{0,+}$ are continuous, not necessary to be differential and are also assumed to be bounded, i.e.,
\begin{align}\label{4}
\max_{t\geq t_0}\max\{h_1(t), h_2(t)\}\leq h_M,
\end{align}
where $h_M$ is a known constant.
The initial condition of system \eqref{1} is given by $x(t_0)=\psi(t_0)$,
$y(s)=\phi(s), \ s \in [t_0-\tau_M,t_0)$. The initial values $\psi(t_0)$ and $\phi(.)$ are unknown but assumed to be bounded by known bounds, i.e.,
\begin{align}
0&\preceq\psi(t_0) \preceq  \overline{\psi}, \label{5}\\
0&\preceq\phi(s) \preceq  \overline{\phi},\ \forall s \in [t_0-h_M,t_0), \label{6}
\end{align}
where $\overline{\psi}, \overline{\phi}$ are known nonnegative constant vectors. Let us denote by $x(t,t_0,\psi,\phi,\omega)$
and $y(t,t_0,\psi,\phi,d)$ the state vectors  with the initial
values ($\psi,\phi$) and disturbances $(\omega(t), d(t))$ of system \eqref{1}.

The objective of this paper is to find as small as possible componentwise upper bounds of  the state vectors of system \eqref{1}. Concisely, we
construct two decreasing vector-valued functions which  are componentwise upper bounds of the two state vectors $x(t,t_0,\psi,\phi,\omega)$ and $y(t,t_0,\psi,\phi,d)$.\\

First, we recall the definition of positive system.
\begin{definition}\label{defn2} \rm (\cite{Kaczorek:02}) System \eqref{1}
is said to be positive if for any non-negative initial values,
$\psi(t_0)\succeq 0$, $\phi(s)\succeq 0,  s \in [t_0-h_M,t_0)$, the state
trajectories of system \eqref{1} satisfy that
$x(t,t_0,\psi,\phi,\omega)\succeq 0$ and $y(t,t_0,\psi,\phi,d)\succeq
0$, for all $t \geq t_0$.
\end{definition}
The following lemmas are needed for our development.
\begin{lemma}\label{lem1}(\cite{Berman:94}). (i) Let $M\in \mathbb{R}^{n\times
n}$ be a nonnegative matrix. Then, the following statements are equivalent:
$(i_1)$ $M$ is Schur stable;
$(i_2)$ $(M-I)q\prec 0$ for some $q \in \mathbb{R}_+^n$;
$(i_3)$ $(I-M)^{-1}\succeq 0$.
\\
(ii) Let $M\in \mathbb{R}^{n\times
n}$ be a Metzler matrix. Then, the following statements are equivalent:
$(ii_1)$ $M$ is Hurwitz stable; $(ii_2)$
$Mq\prec 0$ for some $q \in \mathbb{R}_+^n$;  $(ii_3)$ $M^{-1}\preceq 0$.
\end{lemma}
\begin{lemma}\label{lem2} (\cite{Shen:15,Pathirana:18}) Assume that $A$ is a Metzler matrix,
$B$, $C$, $D$ are nonnegative, $D$ is a Schur matrix. Then, \\

(i) System \eqref{1} is positive.

(ii) For $\psi_1(t_0)\preceq \psi_2(t_0)$ and $\phi_1(s)\preceq
\phi_2(s), s \in [t_0-\tau_M,t_0)$, we have
\begin{align}
x(t,t_0,\psi_1,\phi_1,\omega)&\preceq x(t,t_0,\psi_2,\phi_2,\omega),\ \forall t \geq t_0, \label{7}\\
y(t,t_0,\psi_1,\phi_1,d)&\preceq y(t,t_0,\psi_2,\phi_2,d),\  \forall t \geq t_0. \label{8}
\end{align}
\end{lemma}
\begin{lemma}\label{lem3}(\cite{Ngoc:16})
Assume that $A$ is a Metzler matrix,
$B$, $C$, $D$ are nonnegative,  Then, the following statements are equivalent\\
(i)\ \ $\rho(D)<1$ and $s(A+B(I-D)^{-1}C)<0$;\\
(ii) there exist $p \in \mathbb{R}_+^n$, $q \in \mathbb{R}_+^m$   such that
\begin{align}
&Ap+Bq\prec 0, \label{9}\\
&Cp+(D-I)q\prec 0; \label{10}
\end{align}
(iii) $s(A)<0$ and $\rho(C(-A)^{-1}B+D)<1$.
\end{lemma}
\begin{remark}\label{rem1} \rm Two matrix inequalities \eqref{9}, \eqref{10} can be reformulated into the compact form\\ $
\left[\begin{matrix} A & B\\C & D-I\end{matrix}\right]
\left[\begin{matrix} p\\q\end{matrix}\right]\prec 0$.
By (ii) of Lemma \ref{lem1}, we have $\left[\begin{matrix} A & B\\C & D-I\end{matrix}\right]^{-1}\preceq 0$. Since matrix $\left[\begin{matrix} A & B\\C & D-I\end{matrix}\right]^{-1}$ is non-singular and  all its row vectors  are non-zero. Hence, the vector $[p^T\ \ q^T]^T$ can be computed by
the following equation
\begin{align}\label{11}
\left[\begin{matrix} p\\q\end{matrix}\right]=-\left[\begin{matrix} A & B\\C & D-I\end{matrix}\right]^{-1}\xi,
\end{align}
where $\xi \in \mathbb{R}^{n+m}_{+}$.
\end{remark}
\begin{remark}\label{rem2} \rm The authors \cite{Shen:15,Ngoc:16} used inequalities \eqref{9}, \eqref{10} in order to derive asymptotic stability conditions for CDDEs \eqref{1}. By using  the completeness of the Euclidean space $\mathbb{R}^n$, the authors \cite{Pathirana:18} proposed  tighter inequalities, \eqref{15}, \eqref{16} and \eqref{17}, and used these inequalities to analyze stability of CDDEs with unbounded time-delays. In the paper, we also use these tighter inequalities to derive our main results (componentwise upper bounds of system \eqref{1}). Hence, in the following, we re-state these inequalities and recall its proof.  Assume that condition (i) of Lemma \ref{lem3} holds. Then, by (iii) of Lemma \ref{lem3}, we have $s(A)<0$ which implies that $A^{-1}\preceq 0$ due to (ii) of Lemma \ref{lem1}. Left multiplying $A^{-1}$ on inequality \eqref{9}, we obtain
\begin{align}
-A^{-1}Bq\prec p. \label{12}
\end{align}
Since $D$ is a Schur matrix and nonnegative,  by (i) of Lemma \ref{lem1}, matrix $(I-D)^{-1}$ is nonnegative. Left multiplying $(I-D)^{-1}$ on inequality \eqref{10}, we obtain
\begin{align}
(I-D)^{-1}Cp\prec q. \label{13}
\end{align}
Matrix inequality \eqref{10} is also rewritten as
\begin{align}
Cp+Dq\prec q. \label{14}
\end{align}
Since \eqref{12}, \eqref{13} and \eqref{14} are strict inequalities, by using the completeness of the Euclidean space $\mathbb{R}^n$, there is a positive scalar $\mu\in (0,1)$ such that
\begin{align}
-A^{-1}Bq &\preceq (1-\mu)p\prec p, \label{15}\\
(I-D)^{-1}Cp&\preceq (1-\mu)q\prec q, \label{16}\\
Cp+Dq&\preceq (1-\mu)q \prec q. \label{17}
\end{align}
These tighter inequalities \eqref{15}, \eqref{16} and \eqref{17} will be key estimates for deriving main results in next section.
\end{remark}
\section{Main results}
\subsection{Finite-time convergence of linear positive systems}
In this subsection, we present a method to find the smallest possible time which guarantees the finite-time convergence (i.e., all the state vectors starting from a given bounded set converges within another given bounded set after a finite time) of linear positive systems.
The result is needed in establishing a componentwise upper bound of CDDEs with/without bounded disturbances  in the Sections 3.2 and 3.3. Consider the following linear positive system
\begin{align}\label{18}
\begin{split}
&\dot{u}(t)=Au(t), \ t\geq t_0\geq 0,\\
&u(t_0)=\theta,
\end{split}
\end{align}
where $u(t)\in \mathbb{R}^n_{0,+}$ is the state vector; $A\in \mathbb{R}^{n\times n}$ is a Metzler matrix; the initial value $\theta$ is unknown but  assumed to be bounded by a known bound $\overline{\theta}$, i.e.,
$0\preceq \theta\preceq \overline{\theta}$. Let us denote by $u(t,t_0,\theta)$
 the solution   of system \eqref{18}.

The object of this subsection is,  for a given nonnegative vector $\delta=[\delta_1\ \cdots\  \delta_n]^T \in \mathbb{R}^n_{0,+}$,  to find as small as possible time $T\geq 0$ such that
$u(t,t_0,\theta)\preceq \delta,\ \forall t \geq t_0+T.$\\

\noindent\textit{A.1. Exponential componentwise estimate of linear positive system}

\begin{lemma}\label{lem4}  Assume that $A$ is a Hurwitz stable. Then,
 there are a positive scalar $\alpha>0$ and a vector-value function $\beta(\overline{\theta})$ such that the following exponential componentwise  estimate holds:
\begin{align}\label{19}
u(t,t_0,\theta)\preceq \beta(\overline{\theta})e^{-\alpha (t-t_0)}, \ \forall t\geq t_0.
\end{align}
\end{lemma}
\begin{proof} Since $A$ is Hurwitz stable, there exists a positive scalar $\alpha>0$ such that $A+\alpha I$ is  Hurwitz stable. Note that $s(A^T+\alpha I)=s(A+\alpha I)$. Therefore, matrix $A^T+\alpha I$ is also Hurwitz stable. By (ii) of Lemma \ref{lem1}, there is a vector $v\succ 0$ such that
\begin{align}\label{20}
v^T(A+\alpha I)\prec 0.
\end{align}
 Let us consider the following Lyapunov functional
\begin{align}\label{21}
V(t)=v^Te^{\alpha t}u(t).
\end{align}
By a simple computation, the derivative of $V$ along the solution $u(t,t_0,\theta)$ is given as below
\begin{align}\label{22}
\dot{V}(t)=v^T(A+\alpha I)e^{\alpha t}u(t,t_0,\theta) \leq 0, \ \forall t\geq t_0,
\end{align}
which follows that $V(t)\leq V(t_0),\ \forall t\geq t_0$.
 Combining with $v\succ 0$,  $u(t,t_0,\theta)$ $\succeq 0,\ \forall t \geq t_0$ and $u(t_0,t_0,\theta)=\theta\preceq \overline{\theta}$,
 we have, for each $i\in \overline{1,n}$, that
 \begin{align}\label{23}
v_iu_i(t,t_0,\theta)e^{\alpha t}&\leq v^Tu(t,t_0,\theta)e^{\alpha t}\notag\\
& \leq v^Tu(t_0,t_0,\theta)e^{\alpha t_0}\notag\\
&\leq v^Te^{\alpha t_0}\overline{\theta}, \ \ \forall t\geq t_0,
\end{align}
which follows that
\begin{align}\label{24}
u_i(t,t_0,\theta)\leq \frac{v^T\overline{\theta}}{v_i}e^{-\alpha (t-t_0)}, \ \forall t\geq t_0.
\end{align}
Set $\beta_i(v,\overline{\theta})=\frac{v^T\overline{\theta}}{v_i}$ and $\beta(v,\overline{\theta})=[\beta_1(v,\overline{\theta}),\cdots,
\beta_n(v,\overline{\theta})]^T$. Then, from \eqref{24}, we obtain a exponential componentwise estimate \eqref{19}. The proof of Lemma \ref{lem4} is completed.
\end{proof}
\begin{remark} \label{rem3} \rm For a fixed decay rate $\alpha$ satisfying $s(A^T+\alpha I)<0$. Let us denote by $\Omega$
the set of all vectors $v\succ 0$ such that inequality \eqref{20} holds, i.e.,
\begin{align}\label{25}
\Omega =\{v \in \mathbb{R}_{+}^n: (A^T+\alpha I)v\prec 0\}.
\end{align}
Then, by taking the minimum of  the vector-value function $\beta(v,\overline{\theta})$ subject to $v\in \Omega$, i.e.,
\begin{align}\label{26}
\min_{v\in \Omega}\beta(v,\overline{\theta})= [\min_{v\in \Omega}\beta_1(v,\overline{\theta}), \cdots, \min_{v\in \Omega}\beta_n(v,\overline{\theta})]^T,
\end{align}
we obtain the smallest exponential componentwise  estimate with the fixed decay rate $\alpha$ of system \eqref{18} as below
\begin{align}\label{27}
u(t,t_0,\theta)\preceq \Big(\min_{v\in \Omega}\beta(v,\overline{\theta})\Big)e^{-\alpha (t-t_0)}, \ \forall t\geq t_0.
\end{align}
Next, we present a method to  find $\min_{v\in \Omega}\beta_i(v,\overline{\theta})$, $i\in \overline{1,n}$. For simplicity, we consider the case where $i=1$.
\end{remark}
\noindent\textit{A.2. Minimization of  partial factor $\beta_1(v,\overline{\theta})$}

Let us denote
\begin{align}\label{28}
\Lambda =\{-(A^T+\alpha I)^{-1}r: r\in \ \mathbb{R}_{+}^n\}.
\end{align}
Then, from Lemma \ref{lem1}, we can verify that
\begin{align}\label{29}
\Omega =\Lambda,
\end{align}
which means that every vector $v\in \Omega$ has the following form
\begin{align}\label{30}
v=-(A^T+\alpha I)^{-1}r,
\end{align}
where $r\in \ \mathbb{R}_{+}^n$. By substituting \eqref{30} into formula
$\beta_1(v,\overline{\theta})=\frac{v^T\overline{\theta}}{v_1}$
with some algebraic manipulations,
 $\beta_1(v,\overline{\theta}))$ is simplified into the following rational function of a vector variable $r$, which is denoted by $\Gamma_1(r)$,
\begin{align}\label{31}
\beta_1(v,\overline{\theta})=
\frac{a_1r_1+a_2r_2+\cdots+a_nr_n}{b_1r_1+b_2r_2+\cdots+b_nr_n}\triangleq \Gamma_1(r).
\end{align}
Hence, the problem of finding $\min_{v\in \Omega}\beta_1(v,\overline{\theta})$
 is equivalent to  the problem of finding $\min_{r\in \ \mathbb{R}_{+}^n}\Gamma_1(r)$. Because $\mathbb{R}^n_+$ is an open set in $\mathbb{R}^n$, the minimum of function $\Gamma_1(r)$  subject to $r\in \mathbb{R}^n_+$ may not exist. Therefore, instead of finding
the minimum, we find the infimum of function $\Gamma_1(r)$  subject to $r\in \mathbb{R}^n_+$.

Note that, by Lemma \ref{lem1}, matrix $-(A^T+\alpha I)^{-1}$
is non-negative and nonsingular. This follows that all row vectors of
matrix $-(A^T+\alpha I)^{-1}$
are non-negative and non-zero. Combining with the non-negativeness of vector
$\overline{\theta}$,  we can verify that
 vector $a=[a_1\ a_2 \ \cdots \ a_n]^T$ is non-negative and that vector $b=[b_1\ b_2 \ \cdots \ b_n]^T$ is non-negative and non-zero.
Set $J=\{j\in \overline{1,n}: b_j>0\}$, then by (Lemma 4 in Nam IEEE AC, 2016), we have
\begin{align}\label{32}
\inf_{r\in \mathbb{R}_{+}^{n}}\Gamma_1(r)=\min_{j\in
J}\frac{a_j}{b_j}\triangleq \gamma_1.
\end{align}
Thus,  the smallest exponential estimate with a fixed decay rate $\alpha$ of the $1^{th}$ partial state vector can be given as below:
\begin{align}\label{33}
u_1(t,t_0,\theta)\leq \gamma_1e^{-\alpha (t-t_0)}, \ \forall t\geq t_0,
\end{align}
where $\gamma_1$ is computed by formula \eqref{32}.\\

\noindent\textit{A.3. Finite-time convergence}

For a given scalar $\delta_1>0$, set
 \begin{align}\label{34}
t^1_{\alpha}=\begin{cases} 0 & if \ \  \gamma_1\leq \delta_1,\\
-\frac{1}{\alpha}\ln \frac{\delta_1}{\gamma_1} \ \  & if \ \ \gamma_1>\delta_1.
\end{cases}
\end{align}
From formula \eqref{33} with a simple computation,  we can verify that
\begin{align}\label{35}
u_1(t,t_0,\theta)\leq \delta_1, \ \forall t\geq t_0+t^1_{\alpha}.
\end{align}
Note that $s(A^T+\alpha I)$ is an increasing function with respect to variable $\alpha$. By using the one-dimensional search method, we find the suppremum $\alpha_{max}$ of scalars $\alpha>0$ such that $s(A^T+\alpha I)<0$. Hence, by increasing $\alpha$ gradually  from $0$ to $\alpha_{max}$ with a chosen small step, for example 0.001, and comparing the times $t_{\alpha}^1$ computed by \eqref{34}, we  find
 \begin{align}\label{36}
T^1=\min_{\alpha\in (0,\alpha_{max}]}t_{\alpha}^1.
\end{align}
Then, $T^1$ is the smallest time which guarantees that
\begin{align}\label{37}
u_1(t,t_0,\theta)\leq \delta_1, \ \forall t\geq t_0+T^1.
\end{align}
Similarly, for given scalars $\delta_i>0, i=2,\cdots,n$, we also
compute the smallest times  $T^i,\ i=2,\cdots,n$ such that
$
u_i(t,t_0,\theta)\leq \delta_i, \ \forall t\geq t_0+T^i.
$
Set
\begin{align}\label{38}
T=\max\{T^1,T^2,\cdots, T^n\}.
\end{align}
Then, $T$ is the smallest time, which guarantees that the state vector $u(t,t_0,\theta)$ converges componentwisely within the ball $\mathbb{B}(0,\delta)$ after the finite-time $T$, i.e.,
 \begin{align}\label{39}
u(t,t_0,\theta)\preceq \delta, \ \forall t\geq t_0+T.
\end{align}
 We have now summarized the above presented statements into the following theorem.
\begin{theorem}\label{thm1}Assume that $A$ is a Metzler matrix and Hurwitz stable. Given two vectors $\overline{\theta}\in\mathbb{R}^n_{0,+}$, $\delta\in\mathbb{R}^n_{0,+}$. Then, all trajectories of system \eqref{18} converge componentwisely within the ball $\mathbb{B}(0,\delta)$ after the finite-time, $T$, computed by formula \eqref{38}.
\end{theorem}
\subsection{Componentwise bounds for positive CDDEs without disturbances }
  In this subsection, we present a new result on componentwise bound of  system \eqref{1} for the case no disturbance,  i.e.,
  $\omega(t)\equiv d(t) \equiv 0$. For simplicity, we consider system \eqref{1} with $t_0=0$.
\begin{theorem}\label{thm2}  Assume that $A$ is a Metzler matrix,
$B$, $C$, $D$ are nonnegative, $D$ is a Schur matrix and
$s(A+B(I-D)^{-1}C)<0$. Then, there exist two positive vectors $p \in \mathbb{R}^n_{+}$, $q \in \mathbb{R}^m_{+}$, a scalar $\mu\in (0,1)$, a time $T^*\geq h_M$,  such that, for $k=0,1,2,\cdots,$  the following estimates hold:
 \begin{align}\label{40}
 \begin{split}
x(t,0,\psi,\phi,0)&\preceq  (1-\mu)^kp,\ \  \forall t\in[kT^*,(k+1)T^*), \\
 y(t,0,\psi,\phi,0)&\preceq  (1-\mu)^{k+1}q, \ \  \forall t\in[kT^*,(k+1)T^*).
 \end{split}
\end{align}
\end{theorem}
\begin{proof} \emph{Step 1:} By Lemma \ref{lem3}, Remark \ref{rem1},  there exist two vectors $\widetilde{p} \in \mathbb{R}^{n}_{+}$, $\widetilde{q}\in\mathbb{R}^{m}_{+}$  such that three inequalities \eqref{12}, \eqref{13}, \eqref{14} hold. For  given two vectors $\overline{\psi}\in \mathbb{R}^{n}_{+}$, $\overline{\phi}\in \mathbb{R}^{m}_{+}$, set $\varrho=\max\{\frac{\overline{\psi}_1}{\widetilde{p}_1},\cdots,
\frac{\overline{\psi}_n}{\widetilde{p}_n},
\frac{\overline{\phi}_1}{\widetilde{q}_1},\cdots,
\frac{\overline{\phi}_m}{\widetilde{q}_m}\}$,
and choose $p=\varrho \widetilde{p}$, $q=\varrho \widetilde{q}$. Then,  $p\succeq \overline{\psi}$, $q\succeq \overline{\phi}$ and  \eqref{12}, \eqref{13}, \eqref{14} hold. By using one-dimensional search, we find a scalar $\mu\in (0,1)$ such that inequalities  \eqref{15}, \eqref{16}, \eqref{17} hold. By (ii) of Lemma \ref{lem2}, we have
 \begin{align}\label{41}
 \begin{split}
x(t,0,\psi,\phi,0)&\preceq  x(t,0,p,q,0), \ \forall t \geq 0,\\
 y(t,0,\psi,\phi,0)&\preceq  y(t,0,p,q,0), \ \forall t \geq 0.
 \end{split}
\end{align}
\emph{Step 2:} Next, we prove that there exist a time $T>0$ such that
such that
\begin{align} \label{42}
\begin{split}
x(t,0,p,q,0)&\preceq   (1-\mu)p,\ \forall t \geq T, \\
y(t,0,p,q,0)&\preceq   (1-\mu)^2q, \ \forall t \geq T,
\end{split}
\end{align}
where $p, q$ $\mu$ are computed in Step 1. Indeed, let us consider the linear positive system
\begin{align}\label{43}
\dot{u}(t)&=  Au(t),\ \forall t \geq 0,
\end{align}
and by using Step 1 of the proof of Theorem 1 in \cite{Pathirana:18} and \eqref{15}, \eqref{16}, \eqref{17}, we obtain
\begin{align}
&x(t,0,p,q,0)\preceq p,\ \forall t \geq 0,  \label{44}\\
&y(t,0,p,q,0)\preceq (1-\mu)q,\ \forall t \geq 0, \label{45}
\end{align}
and the following solution comparison
\begin{align}
x(t,0,p,q,0)\preceq -A^{-1}Bq+u(t,0,p+A^{-1}Bq),\forall t \geq 0,  \label{46}
\end{align}
By using Theorem \ref{thm1} for system \eqref{43} with $\overline{\theta}=p+A^{-1}Bq$ and $\delta=(1-\mu)p+A^{-1}Bq$,  we find the smallest time $T$ such that
\begin{align}\label{47}
u(t,0,p+A^{-1}Bq)&\preceq (1-\mu)p+A^{-1}Bq,\forall t \geq T,
\end{align}
which follows that
\begin{align}\label{48}
x(t,0,p,q,0)\preceq (1-\mu)p,\forall t \geq T.
\end{align}
Set $T^*=\max\{T,h_M\}$. Then, from \eqref{45} and \eqref{48}, we have
\begin{align}\label{49}
\begin{split}
&x(t,0,p,q,0)\preceq (1-\mu)p,\ \ \forall t\geq T^*\\
&y(t,0,p,q,0)\preceq (1-\mu)q,\ \ \forall t\geq T^*-h_M.
\end{split}
\end{align}
On the other hand,
from \eqref{1}, we have
\begin{eqnarray}\label{50}
y(t)=\begin{cases} Cx(t)+Dy(t-h_2(t))&\ \text{if}\ \ h_2(t)>0\\
(I-D)^{-1}Cx(t) &\ \text{if}\ \ h_2(t)=0.
\end{cases}
\end{eqnarray}
Combining \eqref{16}, \eqref{17}, \eqref{49} and \eqref{50}, we obtain
\begin{align}\label{51}
y(t,0,p,q,0)\preceq (1-\mu)^2q,\ \ \forall t\geq T^*.
\end{align}
From \eqref{44} and \eqref{45}, we obtain inequality \eqref{40} for the case where $k=0$. From \eqref{49} and \eqref{51}, we obtain inequality \eqref{40} for the case where $k=1$.

\emph{Step 3:} In this step, we prove that inequality \eqref{40} hold for the case where $k=2$. In deed, let us consider the following system with the initial time $t_0=T^*$,
\begin{align}\label{52}
\begin{split}
\dot{x}^1(t)&=Ax^1(t)+By^1(t-h_1(t)),\ t\geq T^*\geq 0, \\
y^1(t)&=Cx^1(t)+Dy^1(t-h_2(t)).
\end{split}
\end{align}
Similar to Step 2, we also prove that
\begin{align}\label{53}
\begin{split}
x^1(t,T^*,p,q,0)&\preceq   (1-\mu)p,\ \forall t \in [2T^*, 3T^*),  \\
y^1(t,T^*,p,q,0)&\preceq   (1-\mu)^2q, \ \forall t \in [2T^*, 3T^*).
\end{split}
\end{align}
By using the linearity of system \eqref{52}, from \eqref{53}, we also obtain,  for any positive scalar $\lambda$, that
\begin{align} \label{54}
\begin{split}
x^1(t,T^*,\lambda p,\lambda q,0)&\preceq  (1-\mu)\lambda p,\ \ \forall t \in [2T^*, 3T^*), \\
y^1(t,T^*,\lambda p,\lambda q,0)&\preceq  (1-\mu)^2\lambda q, \ \ \forall t \in [2T^*, 3T^*).
\end{split}
\end{align}
Choosing $\lambda=1-\mu$ and from \eqref{54}, we have
\begin{align}\label{55}
\begin{split}
 x^1(t,T^*,(1-\mu)p,(1-\mu)q,0)&\preceq  (1-\mu)^2p,\ \ \forall t \in [2T^*, 3T^*),  \\
y^1(t,T^*,(1-\mu) p,(1-\mu)q,0)&\preceq   (1-\mu)^3q, \ \ \forall t \in [2T^*, 3T^*).
\end{split}
\end{align}
On the other hand, inequality \eqref{49} implies that, we have \begin{align}\label{56}
\begin{split}
&x(T^*,0,p,q,0)\preceq (1-\mu)p,\\
&y(t,0,p,q,0)\preceq (1-\mu)q,\ \ \forall t\in [T^*-h_M,T^*).
\end{split}
\end{align}
Combining with part (ii) of Lemma \ref{lem2}, we have
\begin{align}\label{57}
\begin{split}
x(t,0,p,q,0)&\preceq   x^1(t,T^*,(1-\mu)p,(1-\mu)q,0),\ \forall t \geq T^*,  \\
y(t,0,p,q,0)&\preceq   y^1(t,T^*,(1-\mu) p,(1-\mu)q,0), \ \forall t \geq T^*.
\end{split}
\end{align}
From \eqref{55} and \eqref{57}, we obtain
\begin{align}\label{58}
\begin{split}
x(t,0,p,q,0)&\preceq   (1-\mu)^2p,\ \ \forall t \in [2T^*, 3T^*), \\
y(t,0,p,q,0)&\preceq   (1-\mu)^3q,\ \ \forall t \in [2T^*, 3T^*).
\end{split}
\end{align}
This means that we have inequality \eqref{40}  for the case where  $k=2$. By doing similarly as above, we also obtain inequality \eqref{40}  for the cases where  $k=3,4,\cdots.$ The proof of Theorem \ref{thm2} is completed.
\end{proof}
\subsection{Componentwise bounds for positive CDDEs perturbed by
bounded disturbances}
 This subsection is to extend the above obtained result to class of CDDEs perturbed by bounded-nonzero-disturbances. For simplicity, we consider also system \eqref{1} with $t_0=0$.
\begin{theorem}\label{thm3}  Assume that conditions given in Theorem \ref{thm2} hold.  Set
\begin{align}\label{59}
\left[\begin{matrix}\eta\\ \varsigma\end{matrix}\right]
=-\left[\begin{matrix}A & B\\ C &D-I\end{matrix}\right]^{-1}
\left[\begin{matrix}\overline{\omega}\\ \overline{d}\end{matrix}\right].
\end{align}
 (i) There exist two positive vectors $p \in \mathbb{R}^n_{+}$, $q \in \mathbb{R}^m_{+}$, a scalar $\mu\in (0,1)$, a time $T^*\geq h_M$,  such that, for $k=0,1,2,\cdots,$  the following exponential componentwise estimates hold:
 \begin{align}\label{60}
 \begin{split}
x(t,0,\psi,\phi,\omega)&\preceq \eta+ (1-\mu)^kp,\ \  \forall t\in[kT^*,(k+1)T^*), \\
 y(t,0,\psi,\phi,d)&\preceq  \varsigma+(1-\mu)^{k+1}q, \ \  \forall t\in[kT^*,(k+1)T^*).
 \end{split}
\end{align}
(ii) The vector $\left[\begin{matrix}\eta\\ \varsigma\end{matrix}\right]$ is the smallest componentwise ultimate bound of system \eqref{1}.\\

(iii) The ball $\mathbb{B}\left(0,\left[\begin{matrix}\eta\\ \varsigma\end{matrix}\right]\right)$ is the smallest invariant set, which is different from $\{0\}$, of system \eqref{1}.
\end{theorem}
\begin{proof} (i) \ Let us consider the following system:
\begin{align}\label{61}
\begin{split}
\dot{\overline{x}}(t)&=A\overline{x}(t)+B\overline{y}(t-h_1(t))
+\overline{\omega},\ t\geq  0, \\
\overline{y}(t)&=C\overline{x}(t)+D\overline{y}(t-h_2(t))+\overline{d}. \end{split}
\end{align}
Set $
\left[\begin{matrix}\widehat{\psi}\\ \widehat{\phi}\end{matrix}\right]=
\max\left\{\left[\begin{matrix}\overline{\psi}\\ \overline{\phi}\end{matrix}\right],
\left[\begin{matrix}\eta\\ \varsigma\end{matrix}\right]\right\}.$  By both (i) and (ii) of Lemma \ref{lem2}, we have
 \begin{align}\label{62}
 \begin{split}
x(t,0,\psi,\phi,\omega)&\preceq \overline{x}(t,0,\widehat{\psi},\widehat{\phi},\overline{\omega}),\ \ t\geq  0, \\
 y(t,0,\psi,\phi,d)&\preceq  \overline{y}(t,0,\widehat{\psi},\widehat{\phi},\overline{d}), \ \ t\geq 0.
 \end{split}
\end{align}
Taking the following state transformation
 \begin{align}\label{63}
 \begin{split}
\check{x}(t)&=\overline{x}(t)-\eta,\\
\check{y}(t)& =\overline{y}(t)-\varsigma,
\end{split}
\end{align}
then,  we have
\begin{align}\label{64}
\begin{split}
\dot{\check{x}}(t)&=A\check{x}(t)+B\check{y}(t-h_1(t)),\ t\geq  0, \\
\check{y}(t)&=C\check{x}(t)+D\check{y}(t-h_2(t)),
\end{split}
\end{align}
and
 \begin{align}\label{65}
 \begin{split}
\check{x}(t,0,\widehat{\psi}-\eta,\widehat{\phi}-\varsigma,0)&= \overline{x}(t,0,\widehat{\psi},\widehat{\phi},\overline{\omega})-\eta, \\
\check{y}(t,0,\widehat{\psi}-\eta,\widehat{\phi}-\varsigma,0)&= \overline{y}(t,0,\widehat{\psi},\widehat{\phi},\overline{d})-\varsigma, \end{split}
\end{align}
Now, we apply Theorem \ref{lem2} for system \eqref{64} with the initial values $(\widehat{\psi}-\eta,\widehat{\phi}-\varsigma)$, there exist two positive vectors $p,q$, a scalar $\mu\in(0,1)$ and the time $T^*\geq h_M$
such that, for $k=0,1,2,\cdots,$
\begin{align}\label{66}
 \begin{split}
\check{x}(t,0,\widehat{\psi}-\eta,\widehat{\phi}-\varsigma,0)&\preceq (1-\mu)^kp,\ \  \forall t\in[kT^*,(k+1)T^*),  \\
\check{y}(t,0,\widehat{\psi}-\eta,\widehat{\phi}-\varsigma,0)&\preceq (1-\mu)^{k+1}q,\ \  \forall t\in[kT^*,(k+1)T^*).
\end{split}
\end{align}
From \eqref{62}, \eqref{65} and \eqref{66}, we obtain inequalities \eqref{60}.\\

(ii) From \eqref{60}, by letting $t$ tend to infinity, we obtain,
 \begin{align}\label{67}
 \begin{split}
\lim\sup_{t\rightarrow \infty}x(t,0,\psi,\phi,\omega)&\preceq \eta, \\
 \lim\sup_{t\rightarrow \infty}y(t,0,\psi,\phi,d)&\preceq  \varsigma.
 \end{split}
\end{align}
This follows that the vector $\left[\begin{matrix}\eta\\ \varsigma\end{matrix}\right]$ is a componentwise ultimate bound of system \eqref{1}. In order to prove that the vector $\left[\begin{matrix}\eta\\ \varsigma\end{matrix}\right]$ is the smallest
componentwise ultimate bound, we consider the case where $\omega(t)\equiv \overline{\omega}$ and $d(t)\equiv \overline{d}$ and take the following state transformation:
\begin{align}\label{68}
\begin{split}
\widetilde{x}(t)&=\eta-x(t),\\
\widetilde{y}(t)& =\varsigma-y(t),
\end{split}
\end{align}
then, from  \eqref{1}, we obtain
\begin{align}\label{69}
\begin{split}
\dot{\widetilde{x}}(t)&=A\widetilde{x}(t)+B\widetilde{y}(t-h_1(t)),\ t\geq  0, \\
\widetilde{y}(t)&=C\widetilde{x}(t)+D\widetilde{y}(t-h_2(t)),
\end{split}
\end{align}
and
 \begin{align}\label{70}
\begin{split}
\widetilde{x}(t,0,\eta,\varsigma,0)&= \eta-x(t,0,0,0,\overline{\omega}), \\
\widetilde{y}(t,0,\eta,\varsigma,0)&= \varsigma-y(t,0,0,0,\overline{d}). \end{split}
\end{align}
Now, we apply Theorem \ref{lem2} for system \eqref{69} with the initial values $(\eta,\varsigma)$, there exist two positive vectors $p,q$, a scalar $\mu\in(0,1)$ and the time $T^*\geq h_M$
such that, for $k=0,1,2,\cdots,$
\begin{align}\label{71}
\begin{split}
\widetilde{x}(t,0,\eta,\varsigma,0)&\preceq (1-\mu)^kp,\ \  \forall t\in[kT^*,(k+1)T^*),  \\
\widetilde{y}(t,0,\eta,\varsigma,0)&\preceq (1-\mu)^{k+1}q,\ \  \forall t\in[kT^*,(k+1)T^*).
\end{split}
\end{align}
From \eqref{70}, \eqref{71}, we have
\begin{align}\label{72}
\begin{split}
\eta- (1-\mu)^kp      &\preceq x(t,0,0,0,\overline{\omega}),\ \  \forall t\in[kT^*,(k+1)T^*),  \\
\varsigma- (1-\mu)^{k+1}q &\preceq y(t,0,0,0,\overline{d}),\ \  \forall t\in[kT^*,(k+1)T^*).
\end{split}
\end{align}
Letting $t$ tend to infinity, we obtain
\begin{align}\label{73}
\begin{split}
\eta  &\preceq \lim\inf_{t\rightarrow \infty} x(t,0,0,0,\overline{\omega}),\\
\varsigma &\preceq \lim\inf_{t\rightarrow \infty}y(t,0,0,0,\overline{d}).  \end{split}
\end{align}
This implies that the vector $\left[\begin{matrix}\eta\\ \varsigma\end{matrix}\right]$ is the smallest
componentwise ultimate bound. \\

(iii) For the case where $\left[\begin{matrix}\overline{\psi}\\ \overline{\phi}\end{matrix}\right] =
\left[\begin{matrix}\eta\\ \varsigma\end{matrix}\right]$, then
$
\left[\begin{matrix}\widehat{\psi}\\ \widehat{\phi}\end{matrix}\right]=
\left[\begin{matrix}\eta\\ \varsigma\end{matrix}\right]$. Substituting this equality into \eqref{65}, we obtain
\begin{align}\label{74}
\begin{split}
\overline{x}(t,0,\widehat{\psi},\widehat{\phi},\overline{\omega})\equiv\eta, \\
\overline{y}(t,0,\widehat{\psi},\widehat{\phi},\overline{d})\equiv\varsigma.
\end{split}
\end{align}
Combining with \eqref{62}, we obtain
\begin{align}\label{75}
\begin{split}
x(t,0,\psi,\phi,\omega)&\preceq \eta,\ \ t\geq  0, \\
 y(t,0,\psi,\phi,d)&\preceq  \varsigma, \ \ t\geq 0,
 \end{split}
\end{align}
From \eqref{73} and \eqref{75}, we conclude that the ball $\mathbb{B}\left(0,\left[\begin{matrix}\eta\\ \varsigma\end{matrix}\right]\right)$ is the smallest invariant set, which is different from $\{0\}$, of system \eqref{1}.
The proof of Theorem \ref{thm3} is completed.
\end{proof}
 From the above, an algorithm  to compute the
smallest possible componentwise state bound for system \eqref{1} is given as below \\

{\small \centering
\begin{tabular}{l}
\hline
\textbf{Algorithm 1} Computing componentwise state bound \\
\hline
\emph{Step 1:}\ input $A$, $B$, $C$, $D$, $n$, $m$,  $
\overline{\omega}$, $\overline{d}$, $h_M$, $\overline{\psi}$, $\overline{\phi}$,\\
\emph{Step 2:}\ compute $
\left[\begin{matrix}\eta\\ \varsigma\end{matrix}\right]
=-\left[\begin{matrix}A & B\\ C &D-I\end{matrix}\right]^{-1}
\left[\begin{matrix}\overline{\omega}\\ \overline{d}\end{matrix}\right].
$\\
\hspace{1cm} if $\left[\begin{matrix}\overline{\psi}\\ \overline{\phi}\end{matrix}\right]\preceq
\left[\begin{matrix}\eta\\ \varsigma\end{matrix}\right]$,
the componentwise bound is $\left[\begin{matrix}\eta\\ \varsigma\end{matrix}\right]$\\
\hspace{1cm} else, proceed to Step 3\\
\hspace{1cm} end\\
\emph{Step 3:}\ (Finding vectors $p,q$ and $\mu$)\\
\hspace{1cm} $\left[\begin{matrix}\widehat{\psi}\\ \widehat{\phi}\end{matrix}\right]=\max\left\{\left[\begin{matrix}\overline{\psi}\\ \overline{\phi}\end{matrix}\right],
\left[\begin{matrix}\eta\\ \varsigma\end{matrix}\right]\right\}$, replace $\left[\begin{matrix}\overline{\psi}\\ \overline{\phi}\end{matrix}\right]=
\left[\begin{matrix}\widehat{\psi}-\eta\\ \widehat{\phi}-\varsigma\end{matrix}\right]
$,\\
\hspace{1cm} set $\xi=[1 \cdots 1]^T$,  compute
$\left[\begin{matrix}\widetilde{p}\\
\widetilde{q}\end{matrix}\right]=\left[\begin{matrix}A &B\\ C &D-I\end{matrix}\right]^{-1}\xi$, \\
\hspace{1cm} $\varrho=\max\{\frac{\overline{\psi}_1}{\widetilde{p}_1},\cdots,
\frac{\overline{\psi}_n}{\widetilde{p}_n},
\frac{\overline{\phi}_1}{\widetilde{q}_1},\cdots,
\frac{\overline{\phi}_m}{\widetilde{q}_m}\}$,\\
\hspace{1cm}   $p=\varrho \widetilde{p}$,\ $q=\varrho \widetilde{q}$.\\
\emph{Step 4:}\ (Finding $\mu$)\\
\hspace{1cm} Compute $M_1=-A^{-1}Bq$, $M_2=(I-D)^{1}Cp$\\
\hspace{1cm}  $M_3=Cp+Dq$ and obtain\\ \hspace{1cm} $\mu=1-\max\left\{\frac{e_i^TM_1}{e_i^Tp},
\frac{l_j^TM_2}{l_j^Tq}, \frac{l_j^TM_3}{l_j^Tq} \right\}_{i=\overline{1,n},j=\overline{1,m}}$\\
\emph{Step 5:} (Finding $\alpha_{max}$)\\
 \hspace{1cm} $step_1=0.001$, $\alpha=0$ \\
\hspace{1cm} while $\mu(A+\alpha I)< 0$\\
\hspace{2cm} $\alpha=\alpha+step_1$\\
\hspace{1cm} end\\
\hspace{1cm} $\alpha_{max}=\alpha-step_1$.\\
\emph{Step 6:}\ (Finding $T^i,i=1,\cdots,n$ and $T^*$)\\
\hspace{1cm} Set $\overline{\theta}=p+A^{1}Bq$,\ $\delta=(1-\mu)p+A^{1}Bq$, \\
\hspace{1cm} for $i=1:1:n$\\
\hspace{1.2cm} for $\alpha=0:step_1:\overline{\alpha}$\\
\hspace{1.5cm} compute $\gamma_i$ ( by \eqref{32})\\
\hspace{1.5cm} obtain $t^i_{\alpha}$ (by \eqref{34})\\
\hspace{1.2cm} end\\
\hspace{1.2cm} $T^i=\min_{\alpha\in [0,\alpha_{max}]}t^i_{\alpha}$\\
\hspace{1cm} end\\
\hspace{1cm} Obtain  $T=\max\{T^1,\cdots,T^n\}$, $T^*=\max\{T,h_M\}$\\
\hspace{1cm} and the componentwise state bound by \eqref{60}.\\
\hline
\end{tabular}}

\section{Numerical example}
\begin{example}\rm  Consider CDDEs \eqref{1}, whose matrices are chosen as same as ones given in Shen \& Zheng (2015)
\begin{align*}&A=\left[\begin{matrix}
-2.5\ &\ 0.3  &\  0\\
0.5  & -2  &\ 0.1\\
0.4  & 0.6 &\ -3
\end{matrix}\right], \ \
B=\left[\begin{matrix}
0.2  &\ 0.1\\
0.5 & \ 0.3\\
0   & \ 0.4
\end{matrix}\right],
\\
&C=\left[\begin{matrix}
0.3\ &\ 0.4  &\  0.1\\
0.2\ & \ 0.2 & \ 0
\end{matrix}\right], \ \
D=\left[\begin{matrix}
0.6\ \ &0.3\\
0.1\ \ & 0.2
\end{matrix}\right].
\end{align*}
Two disturbance vectors $\omega(t)$, $d(t)$ are bounded by $\overline{\omega}^T=[0.5\ \ 0.3\ \ 0.1]^T$ and $\overline{d}=[0.3 \ \ 0.1]^T$.
The time-varying delays $h_1(t)$, $h_2(t)$ are bounded by $h_M=2$,
The initial values $\psi(0)$ and $\phi(.)$ are bounded by
$\overline{\psi}^T=[2\ \ \ 5 \ \ \ 3]^T$ and $\overline{\phi}^T=[15\ \ \ 5]^T$.\\

By using \eqref{59}, we compute  $\eta^T=[0.7249\ \   1.4756 \ \    0.5780]^T  $, $\varsigma^T=[3.7739\ \ 1.1469]^T$. By choosing $\xi^T=[1 \ \cdots\ 1]^T$ and using Step 1 of the proof of Theorem \ref{thm3}, we find
$p^T=[2.3951 \ \ \ 5.5118\ \ \ 2.4220]^T$,
$q^T=[ 14.1659\ \ \ 4.9990]^T$ and $\mu=0.0707$. By using Step 2 of the proof of Theorem \ref{thm3} and Theorem \ref{thm1}, we find $T=1.2056$. Hence,
$T^*=\max\{T,h_M\}=2$.  As a result, we obtain componentwise state bounds of system \eqref{1} as below:
\begin{align}
\left[\begin{matrix}
x_1(t)\\
x_2(t)\\
x_3(t)
\end{matrix}\right]
&\preceq
\left[\begin{matrix}
0.7249\\
1.4756\\
0.5780
\end{matrix}\right]+
0.9293^k
\left[\begin{matrix}
2.3951\\
5.5118\\
2.4220
\end{matrix}\right],\  \forall t\in  \label{76}\\
&   \in [k2,(k+1)2), \ k=0,1,\cdots, \notag
\end{align}
and
\begin{align}
\left[\begin{matrix}
y_1(t)\\
y_2(t)
\end{matrix}\right]
&\preceq \left[\begin{matrix}
3.7739\\
1.1469
\end{matrix}\right]+
0.9293^{k+1}
\left[\begin{matrix}
14.1659\\
4.9990
\end{matrix}\right],\forall t\in  \label{77}\\
&   \in [k2,(k+1)2), \ k=0,1,\cdots. \notag
\end{align}

For a visual simulation, we choose disturbances vectors as
$$\omega(t)=a\left[\begin{matrix}
0.5|sin(0.2t)|\\
0.3|sin(0.1t)|\\
 0.1|sin(0.3t)|
\end{matrix}\right],\
d(t)=b\left[\begin{matrix}
0.3|cos(0.1t)|\\
0.1|cos(0.2t)|
\end{matrix}\right],$$
where $a\in\{0,\ 0.5,\ 1\}$;  $b\in\{0,\ 1\}$, two time-varying delays as   $h_1(t)=1+|\sin(t)|$,\ $h_2(t)=1+|\cos(t)|$,  and initial values
$\psi=\overline{\psi}$, $\phi=\overline{\phi}$.
The following figures  shows that the trajectories of
  the partial state vectors of system \eqref{1} are bounded by upper bounds computed by Theorem \ref{thm3}.
\begin{figure}[t]
\begin{center}
\begin{minipage}{85mm}
\resizebox*{8.5cm}{!}{\includegraphics{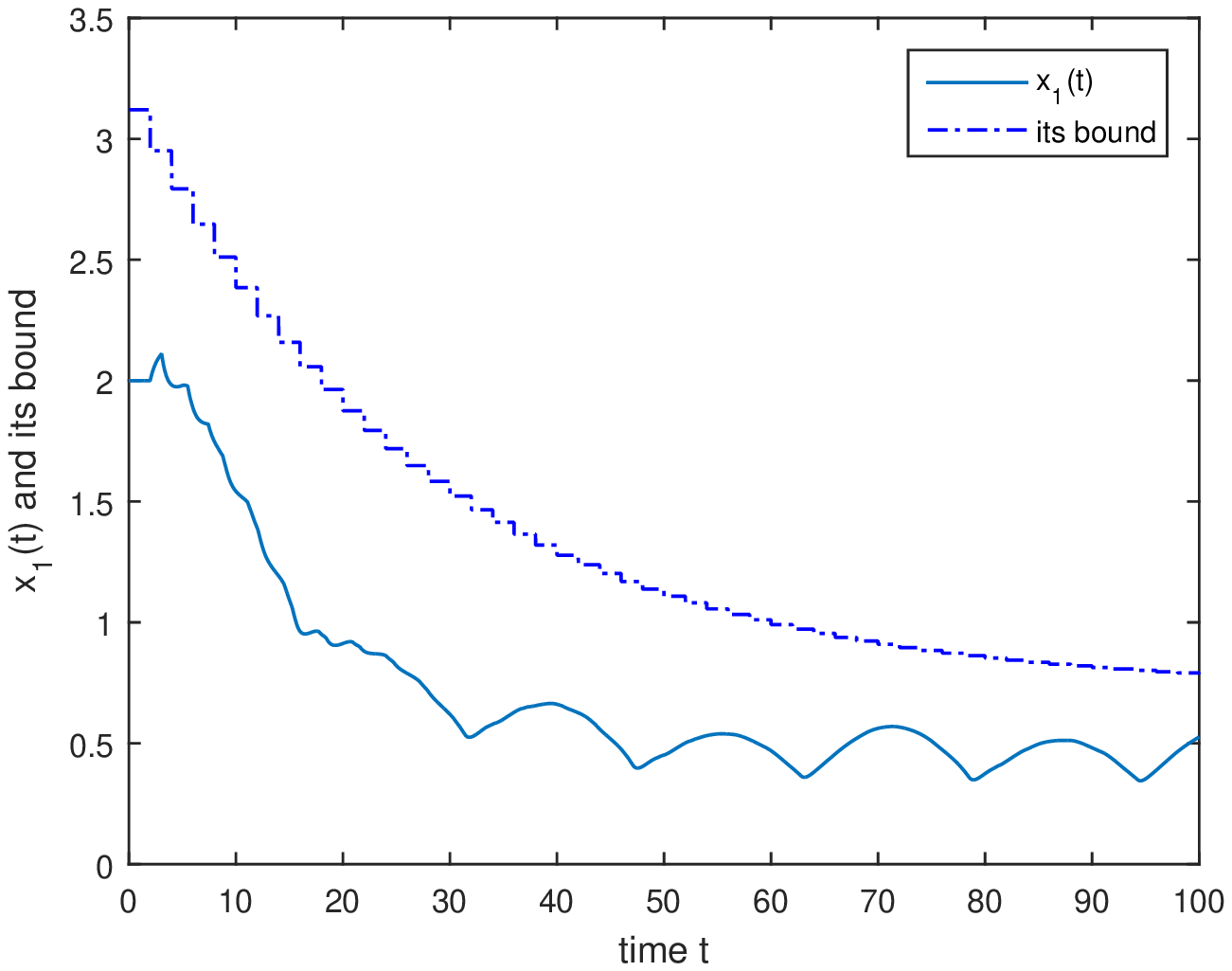}}
\caption{Trajectories of $x_1(t)$ and its bound.}%
\end{minipage}
\begin{minipage}{85mm}
\resizebox*{8.5cm}{!}{\includegraphics{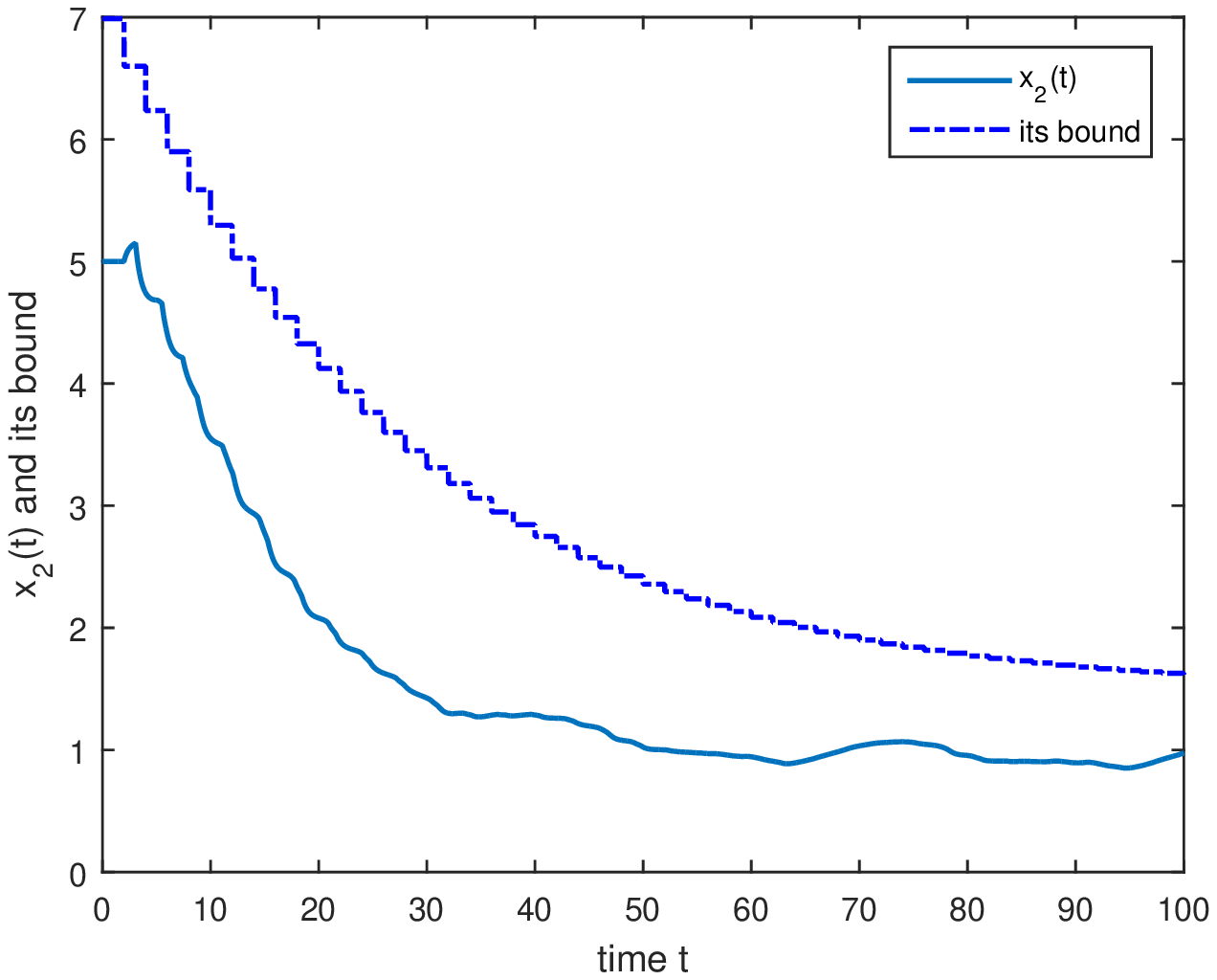}}
\caption{Trajectories of $x_2(t)$ and its bound.}%
\end{minipage}
\begin{minipage}{85mm}
\resizebox*{8.5cm}{!}{\includegraphics{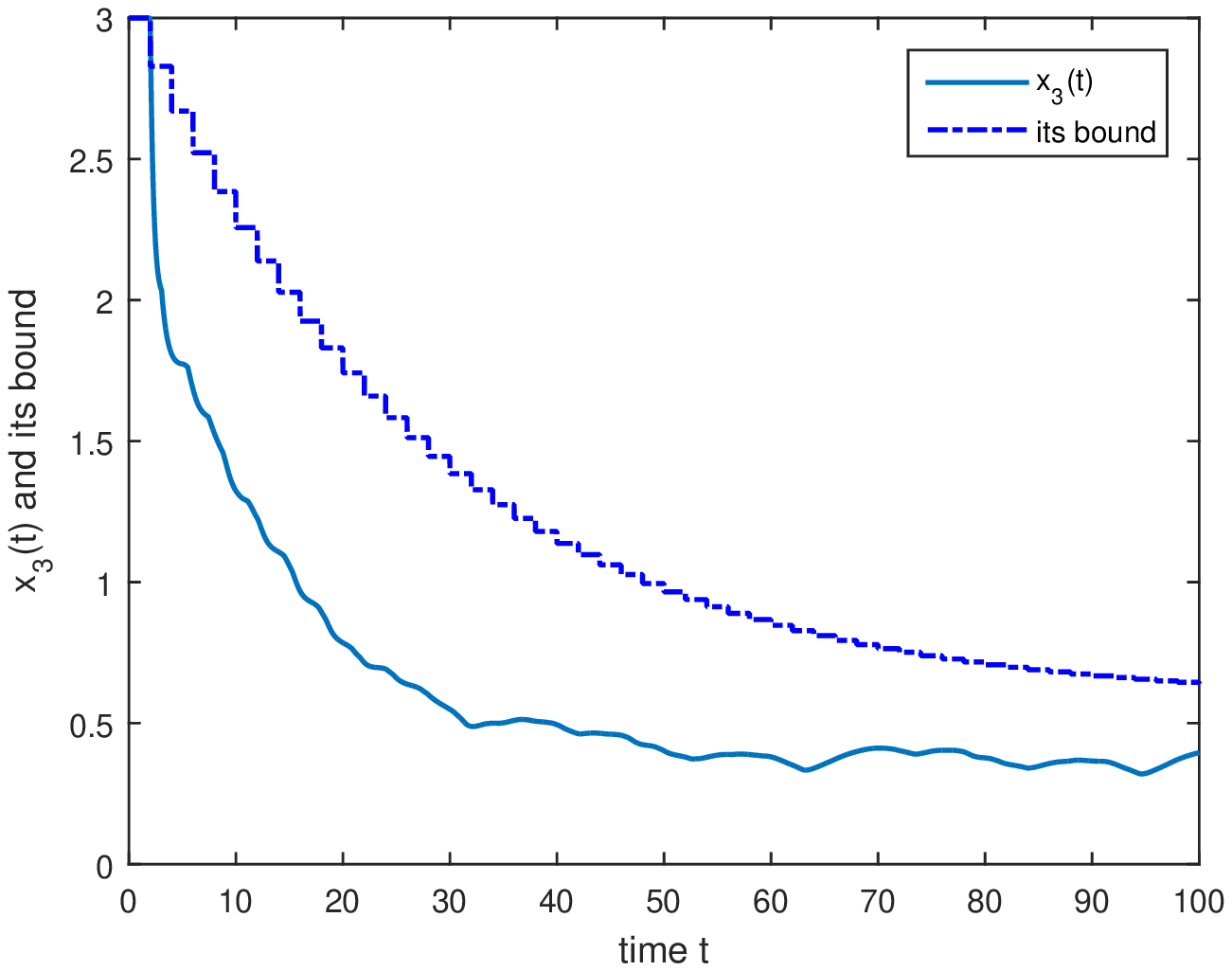}}
\caption{Trajectories of $x_3(t)$ and its bound.}%
\end{minipage}
\end{center}
\end{figure}
\begin{figure}[t]
\begin{center}
\begin{minipage}{85mm}
\resizebox*{8.5cm}{!}{\includegraphics{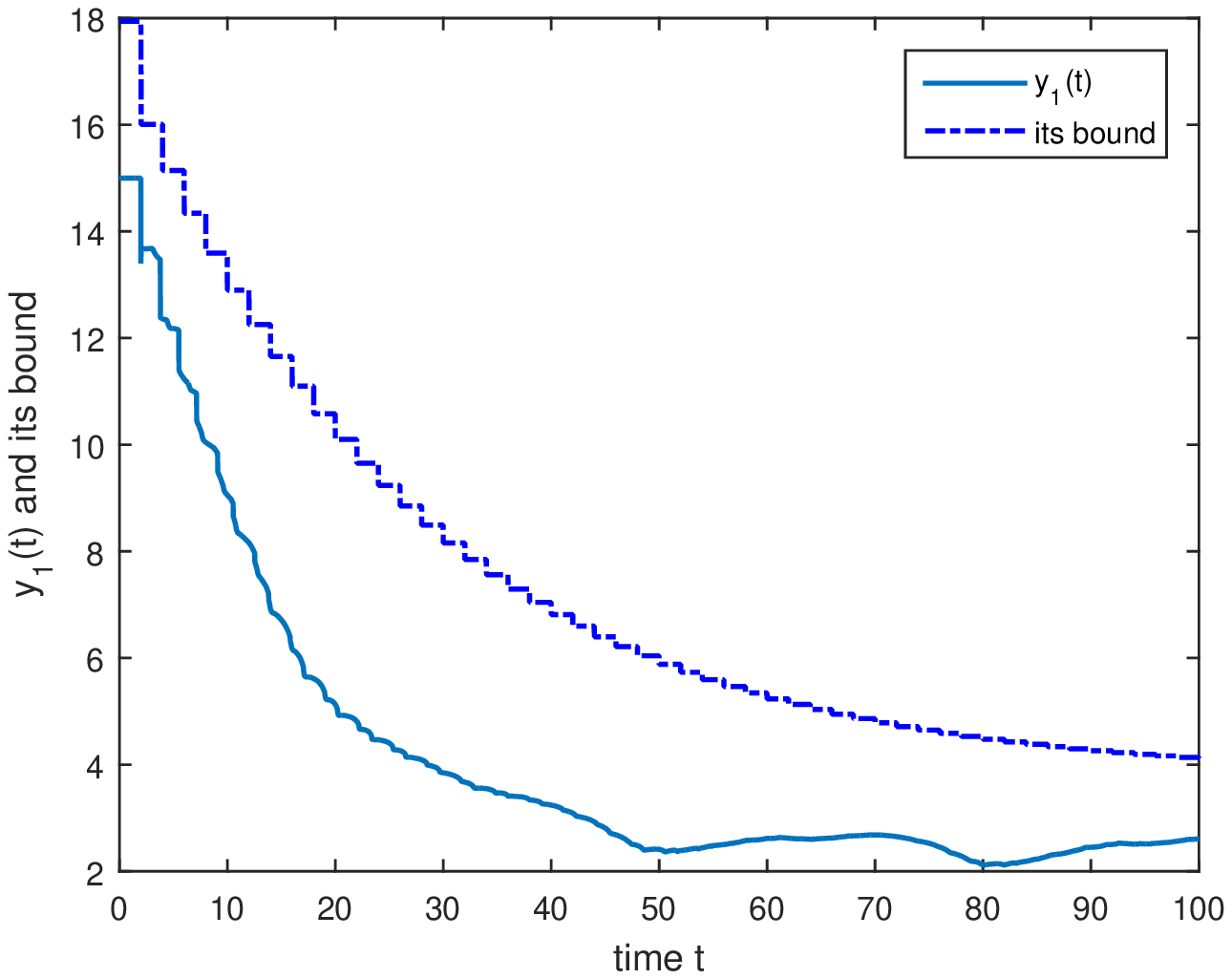}}
\caption{Trajectories of $y_1(t)$ and its bound.}%
\end{minipage}
\begin{minipage}{85mm}
\resizebox*{8.5cm}{!}{\includegraphics{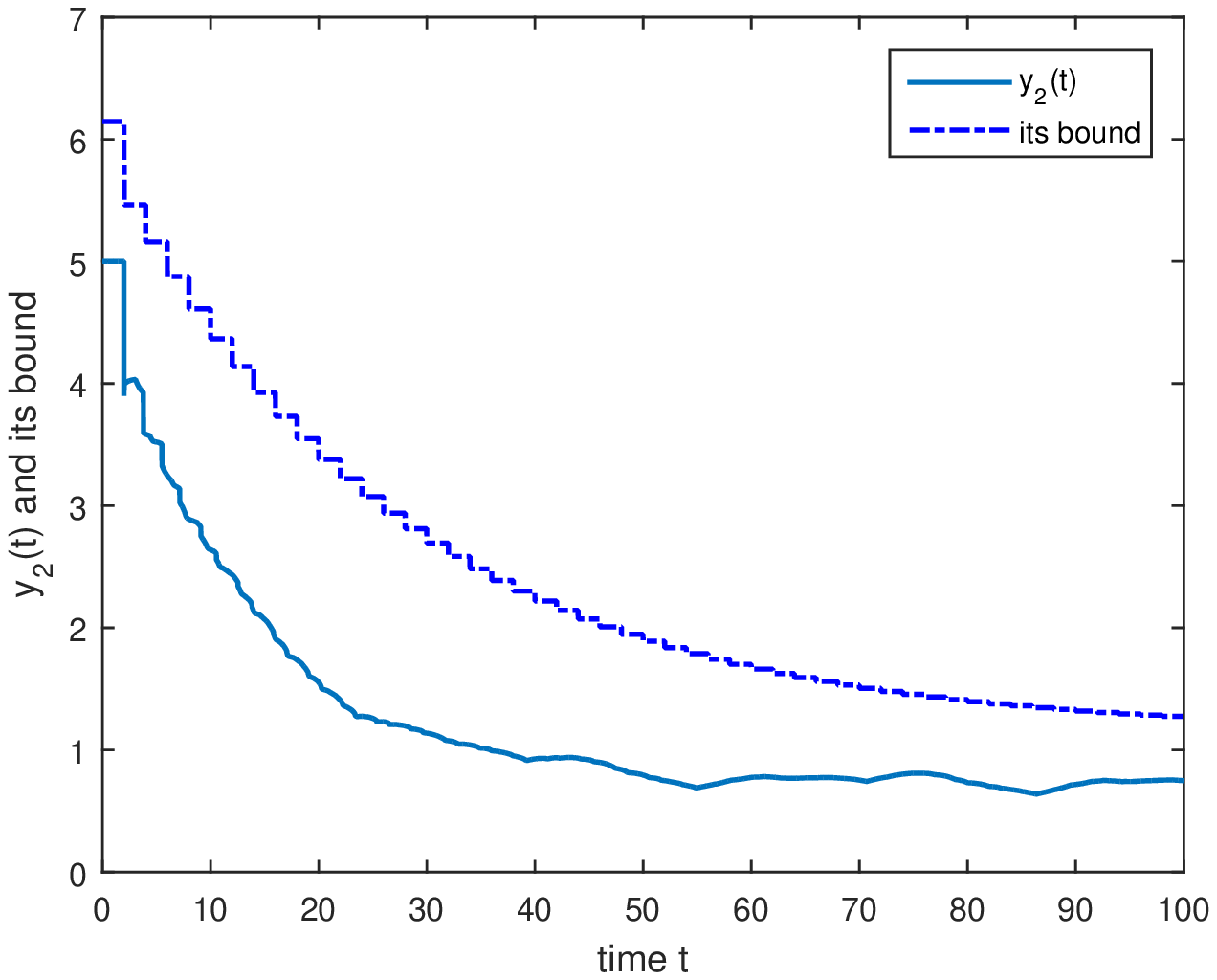}}
\caption{Trajectories of $y_2(t)$ and its bound.}%
\end{minipage}
\end{center}
\end{figure}

%
%

\end{example}
\section{Conclusion}\label{sec10}
This paper has studied the problem of finding state bounds for a class of positive CDDEs perturbed by unknown-but-bounded disturbances. A novel method to derive componentwise state bounds on infinite time horizon, the smallest ultimate bound and the smallest invariant set for the system has been presented.
A numerical example is considered to illustrate the obtained result.
\section{Acknowledgments}\label{sec11}
This work was supported by the Vietnam Institute for Advantaged Study Mathematics
and the National Foundation
for Science and Technology Development, Vietnam under the grant 101.01-2017.300.

\end{document}